\documentclass[12pt,epsfig,amsfonts]{amsart} 
\usepackage{amsmath,amsthm,amssymb,amscd,epsfig}
\usepackage{graphicx}

\newtheorem*{theorem}{Theorem 1}

\newtheorem*{theorema}{Theorem A}
\newtheorem*{theoremb}{Theorem B}

\newtheorem{lemma}{Lemma}[section]
\newtheorem{sublemma}[lemma]{Sublemma}

\newtheorem{prop}[lemma]{Proposition}

\begin{document}
\author{Hiroki Takahasi}

\address{Department of Mathematics,
Keio University, Yokohama,
223-8522, JAPAN} 
\email{hiroki@math.keio.ac.jp}

\title[Prevalence of non-uniform hyperbolicity] 
{Prevalence of non-uniform hyperbolicity \\
at the first bifurcation of H\'enon-like families}

\begin{abstract}
We consider the dynamics of strongly dissipative H\'enon-like maps in the plane, around the first bifurcation parameter $a^*$
at which the uniform hyperbolicity  
is destroyed by the formation of homoclinic or heteroclinic tangencies inside the limit set.
In [Takahasi, H.: Commun. Math. Phys. {\bf 312} 37-85 (2012)], it was proved that
$a^*$ is a full Lebesgue density point of the set of parameters for which 
Lebesgue almost every initial point diverges to infinity under forward iteration.
For these parameters, we show that
all Lyapunov exponents of all invariant ergodic Borel probability measures are uniformly bounded away from zero,
uniformly over all the parameters.
\end{abstract}
\maketitle

\section{introduction}

Hyperbolicity and structural stability are key concepts in the development of the theory of dynamical systems.
Nowadays, it is known that these two concepts are essentially equivalent to each other, at least for $C^1$ diffeomorphisms or flows of a compact manifold
\cite{Hay97, Mane88,Rob71,Rob73, Rob76}.
Then, a fundamental problem in the bifurcation theory 
is to study transitions from hyperbolic to non hyperbolic regimes. 
Many important aspects of this transition are poorly understood.
If the loss of hyperbolicity is due to the formation of a cycle (i.e., a configuration in the phase space involving 
non-transverse intersections between invariant manifolds), an incredibly rich array of complicated behaviors is unleashed by  the unfolding of the cycle
(for instance, see \cite{PalTak93} and the references therein). 

To study bifurcations of diffeomorphisms, we work within a framework set up by Palis:
consider arcs of diffeomorphisms losing their hyperbolicity through generic bifurcations, and
analyze which dynamical phenomena are more frequently displayed (in the sense
of the Lebesgue measure in parameter space) in the sequel of the bifurcation.
More precisely, let $\{\varphi_a\}_{a\in\mathbb R}$ be a parametrized family of diffeomorphisms which undergoes
a first bifurcation at $a=a^*$, i.e., $\varphi_a$ is hyperbolic for $a>a^*$, and
$\varphi_{a^*}$ has a cycle. We assume $\{\varphi_a\}_{a\in\mathbb R}$ unfolds the cycle generically.
A dynamical phenomenon $\mathcal P$ is {\it prevalent} at $a^*$ if
$$\liminf_{n\to\infty}\frac{1}{\varepsilon}{\rm Leb}
\{a\in[a^*-\varepsilon,a^*]\colon\text{$\varphi_a$ displays $\mathcal P$}\}>0,$$
where {\rm Leb} denotes the one-dimensional Lebesgue measure.

Particularly important is the prevalence of hyperbolicity.
The pioneering work in this direction is due to Newhouse and Palis \cite{NP}, on the bifurcation of Morse-Smale diffeomorphisms.
The prevalence of hyperbolicity (or non hyperbolicity) in arcs of surface diffeomorphisms which are not Morse-Smale has been studied in the literature
\cite{MorYoc10, PT0,PT1,PY1,PY2,PY3}. See \cite{Dia95,DiaRoc98,DiaRoc02} for relevant results in higher dimension.
However, for all these and other subsequent developments, including \cite{Rio01, Tak11}, it is fair to say that a global picture is still
very much incomplete.  It has been realized that prevalent dynamics at the first bifurcation considerably depend upon global properties of the diffeomorphisms before or at the bifurcation parameter.

In  \cite{MorYoc10,PT0,PT1,PY1,PY2,PY3},
unfoldings of tangencies of surface diffeomorphisms
associated to basic sets have been treated.
One key aspect of these models   
 is that the orbit of tangency at the first bifurcation is not contained
in the limit set. This implies a global
control on new orbits added to the underlying basic set, and moreover allows one to use
its invariant foliations to translate dynamical problems to the problem on 
how two Cantor sets intersect each other. 
Then, the prevalence of hyperbolicity is related to the Hausdorff dimension of 
the limit set.
This argument is not viable, if the orbit of tangency,
responsible for the loss of the stability of the system, is contained in the limit set. 
Let us call such a first bifurcation an {\it internal tangency bifurcation}.

 In this paper we are concerned with
an arc  $\{f_a\}_{a\in\mathbb R}$ of diffeomorphisms on $\mathbb R^2$ 
of the form \begin{equation}\label{henon}
f_a\colon(x,y)\mapsto(1-ax^2,0)+b\cdot\Phi(a,b,x,y),\quad 0<b\ll1.\end{equation}
Here $\Phi$ is bounded continuous in $(a,b,x,y)$ and 
$C^4$ in $(a,x,y)$. 
This particular arc, often called an ``H\'enon-like family", 
is embedded in generic one-parameter unfoldings of quadratic homoclinic tangencies associated to dissipative saddles
\cite{MorVia93,PalTak93}, and so is relevant
in the investigation of structurally unstable surface diffeomorphisms.


Let $\Omega_a$ denote the non wandering set of $f_a$.
This is an $f_a$-invariant closed set, which 
is bounded (See Lemma \ref{nws}) and so is a compact set. 
It is known \cite{DevNit79} that
for sufficiently large $a>0$, $f_a$ is Smale's horseshoe map
and $\Omega_a$ admits a hyperbolic splitting into uniformly contracting and expanding subspaces.
As $a$ decreases, the infimum of the angles between these two subspaces gets smaller,
and the hyperbolic splitting disappears at a certain parameter.
This first bifurcation is an internal tangency bifurcation. Namely, 
for sufficiently small $b>0$ there exists
a parameter $a^*=a^*(b)$ near $2$ with the following properties \cite{BedSmi04,BedSmi06,CLR08,DevNit79}.:
\begin{itemize}

\item if $a>a^*$, then $\Omega_a$ is a hyperbolic set, i.e.,
there exist constants $C>0$, 
$\xi\in(0,1)$ and at each $x\in\Omega_a$ a non-trivial decomposition $T_x\mathbb R^2=E^s_x\oplus
E^u_x$ with the invariance property such that
$\|D_xf_a^n|E^s_x\|\leq C\xi^{n}$ and $\|D_xf_a^{-n}|E^u_x\|\leq C\xi^{n}$ for every $n\geq0$;

\item 
there is a quadratic tangency $\zeta_0$ near $(0,0)$, between stable and unstable manifolds
of the fixed points of $f_{a^*}$. This tangency is homoclinic when $\det Df_{a^*}>0$ and heteroclinic when 
$\det Df_{a^*}<0$ (See FIGURE 1). 
The orbit of this tangency is accumulated by transverse homoclinic points, 
and so is contained in the limit set.

\end{itemize}

The orbit of tangency of $f_{a^*}$ is in fact unique (See Theorem B),
and $\{f_a\}_{a\in\mathbb R}$ unfolds this unique tangency 
generically. The next theorem gives a partial description of prevalent dynamics at $a=a^*$.

\begin{figure}
\begin{center}
\includegraphics[height=6.5cm,width=14cm]
{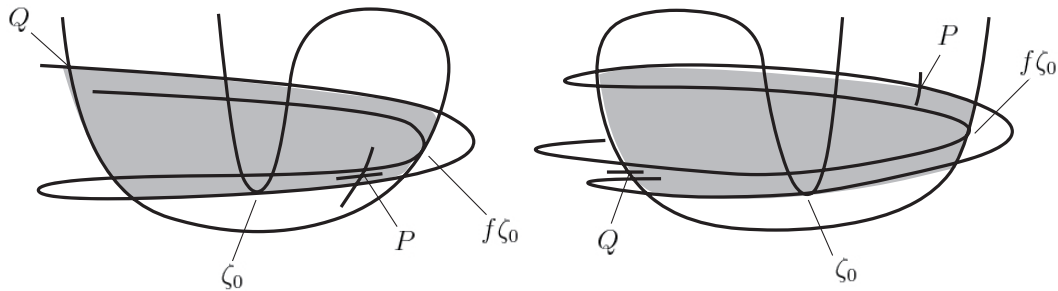}
\caption{Organization of invariant manifolds for $a=a^*$. There exist two
fixed saddles $P$, $Q$ near $(1/2,0)$, $(-1,0)$
respectively. In the case $\det Df>0$ (left),
the stable and unstable manifolds of $Q$ meet each other tangentially. In the case $\det Df<0$
(right), the stable manifold of $Q$ meets the unstable manifold of $P$ tangentially. The shaded
regions represent the region $R$ (see Sect.\ref{local}).}
\end{center}
\end{figure}

\begin{theorem}{\rm (\cite[Theorem]{Tak11})} 
For sufficiently small $b>0$ there exist $\varepsilon_0=\varepsilon_0(b)>0$ and a set $\Delta\subset[a^*-\varepsilon_0,a^*]$ of $a$-values  containing $a^*$ 
with the following properties: 
\begin{itemize}
\item[(a)] $\displaystyle{\lim_{\varepsilon\to+0}} (1/\varepsilon){\rm Leb}(
\Delta\cap[a^*-\varepsilon,a^*])=1;$

\item[(b)] if $a\in\Delta$, then the Lebesgue measure of the set $K_a^+:=\{x\in\mathbb R^2\colon \text{$\{f^n_ax\}_{n\in\mathbb N}$ is bounded}\}$
is zero. In particular, for Lebesgue almost every $x\in\mathbb R^2$,
$|f_a^nx|\to\infty$ as $n\to\infty$.

\end{itemize}
\end{theorem}
In addition, 
if $a\in\Delta$ then $f_a$ is transitive on $\Omega_a$ (See Lemma \ref{K}).
In other words, for ``most" diffeomorphisms immediately right after the first bifurcation,
the topological dynamics is similar to that of Smale's horseshoe before the bifurcation.


The statement and the proof of the above theorem tell us that the dynamics of the diffeomorphisms in $\{f_a\colon a\in\Delta\}$ is fairly structured, and 
this may yield (at least) a weak form of hyperbolicity.
A natural question is the following:
$$\text{{\it To what extent the dynamics is hyperbolic for the diffeomorphisms in $\{f_a\colon a\in\Delta\}$?}}$$
The main result of this paper gives one answer for this question. 
For measuring the extent of hyperbolicity we estimate Lyapunov exponents, the asymptotic exponential rates at which
nearby orbits are separated. 

If there is no fear of confusion, we drop the dependence on $a$ from notation and write $f=f_a$, $\Omega=\Omega_a$ etc.
Let us say that a point $x\in\Omega$ is {\it regular} if there exist number(s) $\chi_1<\cdots<\chi_{r(x)}$
and a decomposition $T_x\mathbb R^2=E_1(x)\oplus\cdots\oplus E_{r(x)}(x)$ such that 
for every $v\in E_i(x)\setminus\{0\}$,
$$\lim_{n\to\pm\infty}\frac{1}{n}\log\Vert D_xf^nv\Vert=\chi_i(x)\quad\text{and}\quad
\lim_{n\to\pm\infty}\frac{1}{n}\log|\det D_xf^n|=\sum_{i=1}^{r(x)}\chi_i(x){\rm dim}E_i(x).$$
By the theorem of Oseledec \cite{Ose68}, the set of regular points has total probability.
If $\mu$ is ergodic, then the functions $x\mapsto r(x)$, $\lambda_i(x)$, ${\rm dim}E_i(x)$ are invariant along orbits,
and so are constant $\mu$-a.e. 
From this and the Ergodic Theorem, 
 one of the following holds for each ergodic measure $\mu$:

\begin{itemize}
\item there exist two numbers $\chi^s(\mu)<\chi^u(\mu)$, and for $\mu$-a.e. $x\in\Omega$ a decomposition
$T_x\mathbb R^2=E^s_x\oplus E^u_x$
such that for any $v^\sigma\in E_x^\sigma\setminus\{0\}$ and
$\sigma=s$, $u$,
$$\lim_{n\to\pm\infty}\frac{1}{n}\log\Vert D_xf^nv\Vert=\chi^\sigma(\mu)\quad\text{and}\quad\int\log|\det Df|d\mu=\chi^s(\mu)+\chi^u(\mu);$$

\item there exists $\chi(\mu)\in\mathbb R$ such that for $\mu$-a.e. $x\in\Omega$
and all $v\in T_x\mathbb R^2\setminus\{0\}$,

$$\lim_{n\to\pm\infty}\frac{1}{n}\log\Vert D_xf^nv\Vert=\chi(\mu)\quad\text{and}\quad\int\log|\det Df|d\mu=2\chi(\mu).$$
\end{itemize}
The number(s) $\chi^s(\mu)$ and $\chi^u(\mu)$, or $\chi(\mu)$ is called a {\it Lyapunov exponent(s)}
of $\mu$.

Let $\mathcal M^e(f)$ denote the set of $f$-invariant Borel probability measures
which are ergodic.
We call $\mu\in\mathcal M^e(f)$ a {\it hyperbolic measure} if $\mu$ has two Lyapunov exponents
 $\chi^s(\mu)$, $\chi^u(\mu)$ with $\chi^s(\mu)<0<\chi^u(\mu)$.
 There is a well-known theory \cite{Kat80,Pes76, Rue79} which allows 
 one to have a fairly good description of the dynamics relative to each hyperbolic measure.
Our main theorem indicates a strong form of non-uniform hyperbolicty for $a\in\Delta$.
\begin{theorema} For sufficiently small $b>0$, the following holds for all
$f\in\{f_a\colon a\in\Delta\}$: \begin{itemize}
\item[(a)]  each $\mu\in\mathcal M^e(f)$ is a hyperbolic measure;

\item[(b)] for each $\mu\in\mathcal M^e(f)$, 
\begin{equation*}\label{them}\chi^s(\mu)<\frac{1}{3}\log b<0<\frac{1}{4}\log2<\chi^u(\mu).\end{equation*}
\end{itemize}
\end{theorema}
It must be emphasized that this kind of uniform bounds on Lyapunov exponents of ergodic measures are compatible with the non hyperbolicity 
of the system, and therefore, Theorem A does not imply the uniform hyperbolicity for $a\in\Delta$.
Indeed, $a^*\in\Delta$ and $f_{a^*}$ is genuinely non hyperbolic, due to the existence of tangencies.
See \cite{CLR08,CLR06} for the first examples of non hyperbolic surface diffeomorphisms of this kind. 
We suspect that the dynamics is non hyperbolic for 
all, or ``most" parameters in $\Delta$.

Little is known on the prevalence of hyperbolicity at internal tangency bifurcations.
The only previously known result in this direction is due to Rios \cite{Rio01},
on certain horseshoes in the plane with three branches. 
However, certain hypotheses in \cite{Rio01} on expansion/contraction rates and curvatures of invariant manifolds
near the tangency, are no longer true for $\{f_a\}_{a\in\mathbb R}$ due to the strong dissipation.

The study of Lyapunov exponents of ergodic measures in the context of homoclinic bifurcations of surface diffeomorphisms traces back to
\cite{CLR08,CLR06}. In higher dimension, the emergence of ergodic measures with zero Lyapunov exponents in unfoldings of heterodimensional cycles
was studied in \cite{BonDiaGor10,DiaGor09}.
For smooth one-dimensional
maps with critical points, the existence of a uniform lower bound on Lyapunov exponents of ergodic measures
is equivalent to several other conditions
 \cite{NowSan98,PrzRivSmi03}, including the Collet-Eckmann Condition which is known to hold for positive Lebesgue measure set of parameters
 \cite{BenCar91}. It would be nice to show more advanced properties of $f_a$, $a\in\Delta$.

\section{Ideas and organization of the proof of Theorem A}

A proof of Theorem A is briefly outlined as follows.
\medskip

\noindent {\it Step 1.} We show that for diffeomorphisms in
$\{f_a\colon a\in\Delta\}$, 
any ergodic measure $\mu$ has two Lyapunov exponents $\chi^s(\mu)<\chi^u(\mu)$ with 
$\chi^s(\mu)<0\leq\chi^u(\mu)$.
Let us call these two numbers a {\it negative} and a {\it nonnegative Lyapunov 
exponent} of $\mu$ respectively.
We show the uniform upper bound on 
negative Lyapunov exponents.
\medskip

\noindent{\it Step 2.} We show the uniform lower bound on nonnegative Lyapunov exponents.
\medskip

Step 1 is fairly easy, and relies on the strong dissipation and the nonexistence of hyperbolic attracting periodic point
for diffeomorphisms in $\{f_a\colon a\in\Delta\}$.
 This is done in Sect.\ref{ups}.

Step 2 is much more involved. As the Oseledec decomposition adapted to a given ergodic measure is not known 
a priori, we analyze the growth of derivatives directly.
All the difficulties come from the folding behavior of the map inside a small fixed neighborhood
$I(\delta)$ of the origin, called a {\it critical region} (See Sect.\ref{initial}).
It is true that, due to the uniform expansion outside of $I(\delta)$ (See Lemma \ref{hyp}), there is a uniform lower bound 
on nonnegative Lyapnov exponents of ergodic measures
whose supports do not intersect $I(\delta)$. 
However, the tangency for $a=a^*$ is accumulated by transverse homoclinic points, 
and thus $I(\delta)$ contains transverse homoclinic points for $a<a^*$ close to $a^*$. 
Then the Poincar\'e-Birkhoff-Smale theorem implies the existence of ergodic measures whose supports intersect $I(\delta)$.
In order to treat nonnegative Lyapunov exponents of these measures, one must treat returns of points to $I(\delta)$.

 We now give a more precise description of Step 2. 
 To treat ergodic measures whose support intersect $I(\delta)$, 
a key ingredient is
 the next proposition which exhausts all possible patterns of growth of derivatives along 
a forward orbit of {\it any} non wandering point. For 
 $x\in \mathbb R^2$ and $n\geq1$ let $w_n(x)=D_{fx}f^{n-1}\left(\begin{smallmatrix}1\\0\end{smallmatrix}\right)$.

\begin{prop}\label{dich}
Let $f\in\{f_a\colon a\in\Delta\}$.
For any  $x\in\Omega$ one of the following holds:
\begin{itemize}
\item[(a)] there exists $\bar\nu\geq0$ such that
$\|w_n(f^{\bar\nu}x)\|\geq e^{\frac{\log2}{4}(n-1)}$
for infinitely many $n\geq1$;

\item[(b)] there exists a sequence $\{\nu_l\}_{l=0}^\infty$ of nonnegative integers such that;

\begin{itemize}

\item[(b-i)] 
$\|w_{\nu_{l+1}}(f^{\nu_0+\cdots+\nu_l}x)\|\geq e^{\frac{\log2}{4}(\nu_{l+1}-1)}$ for every $l\geq0$;

\item[(b-ii)] $\nu_1>0$, and $\nu_{l+1}\geq2\nu_l$ for every $l\geq1$.
\end{itemize}

\end{itemize}
\end{prop}

We now explain how to obtain from Proposition \ref{dich} the desired uniform lower bound on nonnegative Lyapunov exponents.
The next estimate of growth of derivatives is an adaptation of Pesin's result \cite{Pes76}.
It is not particular to the H\'enon-like map $f$ but also holds for any $C^1$ diffeomorphism on a two-dimensional manifold 
admitting an ergodic Borel probability measure with two Lyapunov exponents.
\begin{lemma}\label{zeropesin}
Let $\mu\in\mathcal M^e(f)$ and suppose that $\mu$ has two Lyapunov exponents $\chi^s(\mu)<\chi^u(\mu)$. 
For any $\epsilon>0$ there exists a Borel set $\Lambda(\epsilon)\subset\Omega$ such that 
$\mu(\Lambda(\epsilon))=1$, and
for all $x\in\Lambda(\epsilon)$ there exists $k\in\mathbb N$ such that
for any $v\in T_x\mathbb R^2$, every $m,n\in\mathbb Z$,
\begin{equation}\label{pesineq}
\|D_{f^mx}f^nv\|\leq e^{\epsilon k+\epsilon |m|+\epsilon |n|+\chi^u(\mu)n}\|v\|.
\end{equation}
\end{lemma}

\begin {proof}
Given $\epsilon>0$, for each $k\in\mathbb N$, $k\geq1$,
define $\Lambda_k=\Lambda_k(\epsilon)$ to be the set of points
$x\in \Omega$ for which there is a nontrivial splitting $T_x\mathbb R^2=\tilde E_x^s\oplus \tilde E_x^u$ with the invariance property $D_xf\tilde E^\sigma_x=\tilde E^\sigma_{fx}$
$(\sigma=s,u)$, and the following estimates for every $m,n\in\mathbb Z$:

$$\|D_{f^mx}f^n|\tilde E^s_{f^mx}\|\leq e^{(\epsilon/2) k+(\epsilon/3)|m|+(\epsilon/3)|n|+\chi^s(\mu)n};$$
$$\|D_{f^mx}f^n|\tilde E^u_{f^mx}\|\leq e^{(\epsilon/2) k+(\epsilon/3)|m|+(\epsilon/3)|n|+\chi^u(\mu)n};$$
$$\angle(\tilde E^s_{f^mx},\tilde E^u_{f^mx})\geq e^{-(\epsilon/2) k-(\epsilon/3)|m|}.$$

Set $\Lambda(\epsilon)=\bigcup_{k=1}^\infty\Lambda_k$. It is easy to show that $\Lambda_k$ is a closed set.
Hence, $\Lambda(\epsilon)$ is a Borel set. 
We show $\mu(\Lambda(\epsilon))=1.$ 
From the theorem of Oseledec \cite{Ose68},
for $\mu$-a.e. $x\in \Omega$ and $\sigma=s$, $u$,
$$\lim_{n\to\pm\infty}\frac{1}{n}\log\|D_xf^n|E^\sigma_x\|=\chi^\sigma(\mu)$$
and
$$\lim_{m\to\pm\infty}\frac{1}{m}\log \sin\angle(E^s_{f^mx},E^u_{f^mx})=0.$$

If we fix $\epsilon>0$ then for  each $x$ there exists $N(x)>0$ such that
$$e^{-(\epsilon/3)|n|+\chi^\sigma(\mu)n}\leq\|D_xf^n|E^\sigma_x\|\leq e^{(\epsilon/3) |n|+\chi^\sigma(\mu)n}\ \ 
\text{for every } |n|\geq N(x)$$
and
$$\angle(E^s_{f^mx},E^u_{f^mx})\geq e^{-(\epsilon/3) |m|}\ \ 
\text{for every } |m|\geq N(x).$$
Define $C(x)\geq1$ to be the smallest constant such that 
$$\frac{1}{C(x)}e^{-(\epsilon/3) |n|+\chi^\sigma(\mu)n}\leq\|D_xf^n|E^\sigma_x\|\leq C(x)
e^{(\epsilon/3) |n|+\chi^\sigma(\mu)n}\ \ \text{for every } n\in\mathbb Z$$
and
$$\angle(E^s_{f^mx},E^u_{f^mx})\geq \frac{1}{C(x)}e^{-(\epsilon/2) |m|}\ \ 
\text{for every } |m|\geq N(x).$$
The invariance gives $D_{f^mx}f^n|E^\sigma_{f^mx}=D_{x}f^{m+n}|E^\sigma_{x} \circ (D_{x}f^m|E^\sigma_{x})^{-1}$.
Since the bundle $E^\sigma$ is one-dimensional,
we have
\begin{align*}\|D_{f^mx}f^n|E^\sigma_{f^mx}\|&=\|D_{x}f^{m+n}|E^\sigma_{x}\| \cdot \|(D_{x}f^m|E^\sigma_{x})^{-1}\|\\
&=\frac{\|D_{x}f^{m+n}|E^\sigma_{x}\|}{\|D_{x}f^m|E^\sigma_{x}\|}\leq C(x)^2e^{(2\epsilon/3)|m|+(\epsilon/3) |n|+\chi^\sigma(\mu)n}\\
&\leq \frac{1}{3\pi}e^{(\epsilon/2) k+(2\epsilon/3)|m|+\epsilon |n|+\chi^\sigma(\mu)n},\end{align*}
and
$$\angle(E^s_{f^mx},E^u_{f^mx})\geq e^{-(\epsilon/2) k}e^{-(\epsilon/3)|m|},$$
provided $C(x)^2\leq (1/3\pi)e^{(\epsilon/2) k}$.
Hence $x\in\Lambda_k$ with $\tilde E_{f^mx}^\sigma=E_{f^mx}^\sigma$. 
This yields $\mu(\Lambda(\epsilon))=1.$

 We prove \eqref{pesineq}.
 Let $x\in\Lambda$ and $k\in\mathbb N$ be such that $x\in\Lambda_k$. 
 Take a unit vector $e^\sigma_x$ 
spanning $E_x^\sigma$ ($\sigma=s,u$) so that $\langle e^s_x,e^u_x\rangle>0$,
where the bracket denotes the standard inner product.
Let $v\in T_{f^mx}\mathbb R^2$ be a unit vector.
 Split $v=\xi^ue^u_x+\xi^se^s_x$. 
 It is not hard to see
  $$\max\{|\xi^s|,\ |\xi^u|\}\leq\frac{3\pi}{2}\cdot\frac{1}{(\angle(E_{f^mx}^s,E_{f^mx}^u))^2}.$$
 Using the above estimates
 and  the assumption $\chi^s(\mu)<\chi^u(\mu)$ we obtain
 \begin{align*}\|D_{f^mx}f^nv\|
 &\leq|\xi^s|\cdot\|D_{f^mx}f^n|E^s_{f^mx}\|+|\xi^u|\cdot\|D_{f^mx}f^n|E^u_{f^mx}\|\\
&\leq  \frac{1}{2}e^{\epsilon k+\epsilon|m|+\epsilon |n|+\chi^s(\mu)n}+
\frac{1}{2}e^{\epsilon k+\epsilon|m|+\epsilon |n|+\chi^u(\mu)n}\\
&< e^{\epsilon k+\epsilon|m|+\epsilon |n|+\chi^u(\mu)n}. \qedhere\end{align*}
\end{proof}
\medskip

\noindent{\it Remark.} 
In Lemma \ref{zeropesin} we do not assume $0\notin\{\chi^s(\mu),\chi^u(\mu)\}$.
\medskip

Returning to the proof of Theorem A,
let $f\in\{f_a\colon a\in\Delta\}$ and $\mu\in\mathcal M^e(f)$.
Fix $\epsilon>0$. Consider a point $x\in\Lambda(\epsilon)$.
We first treat the case where Proposition \ref{dich}(a) holds for $x$.
Then, for infinitely many $n\geq1$ we have
$$e^{\epsilon k}e^{\epsilon(\bar\nu+1)}e^{(\chi^u(\mu)+\epsilon)(n-1)}\geq
\|w_n({f^{\bar\nu}x})\|\geq e^{\frac{\log2}{4}(n-1)},$$
where the first inequality follows from Lemma \ref{zeropesin},
and the second from Proposition \ref{dich}(a).
Taking logs of both sides and rearranging the result gives
$$\chi^u(\mu)\geq\frac{1}{4}\log2-\epsilon-\epsilon\cdot \frac{k+\bar\nu+1}{n-1}.$$
Letting $n\to\infty$ we get
\begin{equation}\label{low1}\chi^u(\mu)\geq\frac{1}{4}\log2-\epsilon.\end{equation}

We now treat the case where Proposition \ref{dich}(b) holds.
Then, there exists a sequence $\{\nu_l\}_{l=0}^\infty$ of nonnegative integers such that
$$
e^{\epsilon k}e^{\epsilon(\nu_0+\cdots+\nu_l+1)}e^{(\chi^u(\mu)+\epsilon)(\nu_{l+1}-1)}\geq
\|w_{\nu_{l+1}}(f^{\nu_0+\cdots+\nu_l}x)\|\geq e^{\frac{\log2}{4}(\nu_{l+1}-1)},$$
where the first inequality follows from Lemma \ref{zeropesin} and the second from Proposition \ref{dich}(b).
Taking logs of both sides and rearranging the result gives
$$\chi^u(\mu)\geq\frac{1}{4}\log2-\epsilon-\epsilon\cdot \frac{k+\nu_0+\cdots+\nu_l+1}{\nu_{l+1}-1}.$$
Since $\nu_1>0$ and $\nu_{l+1}\geq2\nu_l$ for every $l\geq1$,
$$\nu_1+\cdots+\nu_l\leq \nu_{l+1}\sum_{i=1}^l2^{-i}\leq \nu_{l+1}.$$
There exists $l_0=l_0(k,\nu_0)$ such that if $l\geq l_0$, then
$$ \frac{k+\nu_0+\cdots+\nu_l+1}{\nu_{l+1}-1}\leq 2.$$
Plugging this into the previous inequality yields
\begin{equation}\label{low2}\chi^u(\mu)\geq\frac{1}{4}\log2-3\epsilon.\end{equation}
Since $\epsilon>0$ can be chosen arbitrarily small, from \eqref{low1} \eqref{low2} we obtain 
$$\chi^u(\mu)\geq\frac{1}{4}\log2.$$
Since $\mu$ is arbitrary, we obtain the uniform lower bound on nonnegative Lyapunov exponents.
\medskip

\noindent{\it Remark.} Since $\Phi$ in \eqref{henon} is assumed to be bounded,
the H\'enon family $H_a\colon(x,y)\mapsto (1-ax^2-\sqrt{b}y,\pm\sqrt{b}x)$
 does not have the form in \eqref{henon}.
However,
from \cite[Proposition 2.1]{BedSmi06} there exists a square which contains the non wandering 
set of $H_a$ with $(a,b)$ close to $(2,0)$. Hence, one can modify $H_a$
outside of the square so that the resultant family has the form in \eqref{henon}.
As a result, the same statements as in Theorem A hold for the H\'enon family as well.
\medskip

The rest of this paper consists of two sections.
In Sect.3 we prove Proposition \ref{dich}, and complete the proof of Theorem A. 
In Sect.4 we 
show that the tangency at $a=a^*$ is unique,
in the sense that any homoclinic or heteroclinic point other than $f_{a^*}^n\zeta_0$ $(n\in\mathbb Z)$ is transverse (Theorem B).  

\section{Proof of Theorem A}
In this section we finish the proof of Theorem A. 
 In Sect.\ref{ups} we obtain
the desired uniform upper bound on negative Lyapunov exponents.
In Sect.\ref{local} we define a compact domain containing the non wandering set,
and use it to show the transitivity (Lemma \ref{K}).
 The rest of this section is entirely dedicated to the proof of  Proposition \ref{dich}.

\subsection{Uniform upper bound on negative Lyapunov exponents}\label{ups}
We say $p\in\mathbb R^2$ is a {\it periodic point} of $f$ if there exists $n>0$ such that
$f^np=p$. The smallest $n$ with this property is denoted by $\pi(p)$ and 
 called the {\it period} of $p$. 
A periodic point $p$ is called {\it hyperbolic attracting} if 
 all the eigenvalues of 
$D_pf^{\pi(p)}$ are strictly contained in the unit circle.



\begin{lemma}\label{simple}
If there is no hyperbolic attracting periodic point of  $f$, then any $\mu\in\mathcal M^e(f)$
has two Lyapunov exponents $\chi^s(\mu)<\chi^u(\mu)$.
In addition, $\chi^s(\mu)\leq (1/3)\log b<0\leq\chi^u(\mu).$
\end{lemma}

\begin{proof}
From the proof of  \cite[Theorem 4.2]{Kat80} we know that ergodic measures whose Lyapunov exponents
are all negative are supported on orbits of hyperbolic attracting periodic points.
Hence, under the assumption of Lemma \ref{simple}, any $\mu\in\mathcal M^e(f)$ has at least one nonnegative Lyapunov exponent.

If $\mu$ has only one Lyapunov exponent, then
$\chi(\mu)=(1/2)\int\log|Df|d\mu<0$, a contradiction. Hence, $\mu$ has two Lyapunov exponents $\chi^s(\mu)<\chi^u(\mu)$.
Since $|\det Df|\leq Cb$ for some $C>0$ independent of $b$, we have
$$2\chi^s(\mu)\leq \chi^s(\mu)+\chi^u(\mu)=\int\log|\det Df|d\mu\leq\log C+\log b,$$
which yields the desired inequality inequality for sufficiently small $b$. 
 \end{proof}

Diffeomorphisms in $\{f_a\colon a\in\Delta\}$ has no hyperbolic attracting periodic point,
for otherwise the Lebesgue measure of the set $K_a^+$ is positive.
Hence, Lemma \ref{simple} yields the desired uniform upper bound on negative Lyapunov exponents 
of ergodic measures.

\subsection{The non wandering set}\label{local}
A periodic point $p\in\mathbb R^2$ is called a \emph{saddle} if one eigenvalue of $D_pf^{\pi(p)}$
has norm bigger than one and the other smaller than one.
For a saddle $p$, denote by $W^s(p)$ and $W^u(p)$ its stable and unstable manifolds
respectively.

For $(a,b)\in\mathbb R^2$ close to $(2,0)$, $f_a$ may be viewed as a singular perturbation of 
the endomorphism $(x,y)\mapsto(1-2x^2,0)$, having exactly two fixed points
$P=(1/2,0)$, $Q=(-1,0)$ which are repelling. Hence 
$f_a$ has exactly two fixed saddles close to $P$ and $Q$, denoted by
 $P_a$ and $Q_a$ respectively.
 If there is no fear of confusion,
we write $P=P_a$, $Q=Q_a$ 
with a slight abuse of notation.
If $\det Df>0$, then let $W^u=W^u(Q)$. 
If $\det Df<0$, then let $W^u=W^u(P)$.

Since any invariant probability measure is supported on a subset of the non wandering set, 
the next lemma allows us to restrict our consideration to a certain compact domain.
\begin{lemma}\label{nws}
For sufficiently small $b>0$ there exists $\varepsilon_0=\varepsilon_0(b)>0$ such that
for all
 $f\in\{f_a\colon a\in[a^*-\varepsilon_0,a^*]\}$ there exists a compact domain $R$  
 located near $\{(x,0)\in\mathbb R^2\colon |x|\leq1\}$, bordered by two curves in $W^s(Q)$ and two in $W^u$,
with the following properties:

\begin{itemize}

\item[(a)] If $x\in R$ and $fx\notin R$, then $f^{n}x\notin R$ for every $n>1$ and $|f^nx|\to+\infty$ as 
$n\to+\infty$;

\item[(b)] $\Omega\subset R$.
\end{itemize}
\end{lemma}

\begin{proof}
Let $V=\{(x,y)\in\mathbb R^2\colon |y|\leq\sqrt{b}\}$.
For $a\in\mathbb R$ close to $a^*$, let $A_0^-=A_0^-(a)$ denote the component of $V\cap W^s(Q)$ containing $Q$. 
Let $A_0^+$ denote the component of $V\cap f^{-1}A_0^-$ not intersecting $A_0^-$.
The $A_0^\pm$ are nearly vertical curves.
Let $\hat R$ denote the compact domain bordered by $\partial V$ and $A_0^\pm$.
Let $\hat\gamma$ denote the compact curve in $W^u$ containing the saddle in $W^u$
and having endpoints in $A_0^-$ and $A_0^+$.
 Define $R$ to be the compact domain bordered by $A_0^-$, $A_0^+$ and the two curves in $f\hat \gamma$
 intersecting both $A_0^-$ and $A_0^+$.

By definition, the set $\hat R\setminus R$ has two components. Let $V_1$ denote the component
whose boundary contains the fixed point in $W^u$. Let $V_2$ denote the other component.
\begin{sublemma}\label{accumulate}
If $x\in V_1$, then $f^{-n}x\notin V$ holds for some $n>0$.
\end{sublemma}

\begin{proof}
The $V_1$ is bordered by $A_0^\pm$, $\hat\gamma$, and a horizontal segment,
denoted by $\gamma_0$.
Define inductively $\gamma_n=\hat R\cap f\gamma_{n-1}$ for $n\geq1$.
The Inclination Lemma implies that $\gamma_n$ accumulates on $\hat \gamma$.
Since the stable eigenvalue of the saddle in $W^u$ is positive,
this accumulation takes place from one side. Hence the statement holds.
\end{proof}

Sublemma \ref{accumulate} and 
$V\cap f^{-1}V_2\subset V_1$ together imply that 
any point in $\hat R\setminus R$ is mapped by some backward iterates of $f$
to the outside of $V$. Since $\Omega=f\Omega\subset f\mathbb R^2\subset V$,
$(\hat R\setminus R)\cap\Omega=\emptyset$ holds.

Let $x\in V\setminus\hat R$. From the form of our map \eqref{henon},
there is an open neighborhood $U$ of $x$ such that
the first coordinate of any point in $f^nU$ goes to $-\infty$ as $n\to\infty$.
Hence
$U\cap f^nU\neq\emptyset$ holds for every $n\in\mathbb N\setminus\{0\}$,
and so $x\notin\Omega$.
Consequently we obtain $\Omega\subset R$.
\end{proof}

\begin{lemma}\label{K}
If $f\in\{f_a\colon a\in\Delta\}$, then $f$ is transitive on $\Omega$.
\end{lemma}

\begin{proof}  
We show both $W^s(Q)$ and $W^u$ are dense in the set $K:=\{x\in\mathbb R^2\colon \text{$\{f^nx\}_{n\in\mathbb Z}$ is bounded} \}$.
Since the stable and unstable manifolds of the 
fixed points of $f$ have mutual transverse intersections, from the Inclination Lemma it follows that
$f$ is transitive on $K$ and
$K\subset\Omega$. The reverse inclusion is a consequence of Lemma \ref{nws}(b).
As a corollary, $f$ is transitive on $\Omega$.

Recall that $K^+=\{x\in\mathbb R^2\colon \text{$\{f^nx\}_{n\in\mathbb N}$ is bounded} \}$ (See Theorem 1). 
Let $x\in K$, and $U$ be an open set containing $x$. 
Let $\Theta_\infty$ denote the compact lenticular domain bordered by the parabola in $W^s(Q)$ and one of the two boundary curves of
$R$ formed by $W^u$. Since the Lebesgue measure of 
$U\cap K^+$ is zero and $f^{-1}(\mathbb R^2\setminus R)\cap R={\rm int}\Theta_\infty$, there exists $n\geq0$ such that
$f^nU$ intersects the parabola in $W^s(Q)$.
Hence $U\cap W^s(Q)\neq\emptyset$. Since $x$ and $U$ are arbitrary, it follows that $W^s(Q)$ is dense in $K$.  

Consider $f^nR$ for $n\geq0$. Its boundary consists of segments of $W^s(Q)$ and $W^u$. 
The segments of $W^s(Q)$ become shorter and converge to $Q$ for increasing $n$.
Moreover, the area of $f^nR$ goes to zero as $n$ increases.
From the proof of Lemma \ref{nws}, if $x\in K$ then $f^{-n}x\in R$ holds for every $n\geq0$.
Consequently,
for any $x\in K$ and $n>0$ large, $x$ is near the boundary
of $f^nR$, and hence, near $W^u\cup W^s(Q)$. Since the part of the boundary formed by $W^s(Q)$
has decreasing length, $x$ is close to $W^u$. It follows that $W^u$ is dense in $K$.
\end{proof}

\subsection{Preliminaries for the proof of Proposition \ref{dich}}\label{initial}
To obtain the desired uniform lower bound on nonnegative Lyapunov exponents,
we estimate the growth of derivatives from below. 
One key observation is that nearly horizontal vectors grow exponentially fast in norm,
 as long as orbits stay out of a small critical region.

Set 
$\hat\lambda:=\frac{99}{100}\log 2$. For $\delta>0$ set $I(\delta):=\{(x,y)\in R\colon |x|<\delta\}.$
For a tangent vector $v=\left(\begin{smallmatrix}\xi\\ \eta\end{smallmatrix}\right)$ with $\xi\neq0$,
let $s(v):=|\eta/\xi|$. 
By a {\it $C^2(b)$-curve} we mean a compact, nearly horizontal $C^2$-curve in $\mathbb R^2$ such that the slopes of its tangent directions
are $\leq\sqrt{b}$ and the curvature is everywhere $\leq \sqrt{b}$.

\begin{lemma}\label{hyp}
For any $\delta>0$ 
 there exists an open set $\mathcal U\subset\mathbb R^2$ containing $(2,0)$ such that 
if $(a,b)\in\mathcal U$ then the following holds for $f=f_a$:
\begin{itemize}
\item[(a)]  if $n\geq1$ and $x\in  \mathbb R^2$ are such that $x,fx,\ldots,f^{n-1}x \in R\setminus I(\delta)$,
then for any nonzero vector $v$ at $x$ with $s(v)\leq\sqrt{b}$, 
$\|D_xf^nv\|\geq \delta e^{\hat\lambda n}\|v\|$ and $s(D_xf^nv)\leq\sqrt{b}$. If, in addition $f^nx\in I(\delta)$, then $\|D_xf^nv\|\geq e^{\hat\lambda n}\|v\|$;

\item[(b)] if $\gamma\subset R\setminus I(\delta)$ is a $C^2(b)$-curve, then so is $f\gamma$.

\end{itemize}

\end{lemma}
\begin{proof}
Since $a^*\to2$ as $b\to0$, $f$ may be viewed as a singular perturbation of 
the Chebyshev quadratic $x\in\mathbb R\mapsto 1-2x^2$, which is smoothly conjugate to the tent map. 
From this (a) follows. 
 (b) follows from (a) and \cite[Lemma 2.4]{WanYou01}. 
\end{proof}

We handle returns to the inside of $I(\delta)$ with the help of {\it critical points}.
The parameters in $\Delta$ correspond to maps for which critical points are well-defined.
The critical points are constructed by an inductive step, and this construction has
to take into consideration possible escapes from $R$ under positive
iteration. Since the unbounded derivatives of \eqref{henon} at infinity 
is problematic, we modify $f$ outside of a fixed neighborhood of $R\cup fR$
so that derivatives are uniformly bounded on $\mathbb R^2$. Precise requirements
are the following (\cite[pp.41]{Tak11}).

\begin{itemize}

\item for any $x\in fR\setminus R$ and every $n\geq1$, $f^nx\notin R$;

\item for any $x\in fR\setminus R$, $v\in T_x\mathbb R^2\setminus\{0\}$
with $s(v)\leq\sqrt{b}$ 
and every $n\geq1$, $s(D_xf^nv)\leq\sqrt{b}$ and $\|D_xf^nv\|\geq 2^n\|v\|;$

\item there exists a constant $C_0>0$ independent of $b$ such that
$|\det Df|\leq C_0b$ and $\|\partial^if\|\leq C_0$ on $\mathbb R^2$ $(1\leq i\leq 4)$, 
where $\partial^i$ denotes any of the partial derivatives of order $i$
on $(a,x,y)$.

\end{itemize}

The following constants $$\lambda:=\frac{99}{100}\hat\lambda,\ \alpha,\ \delta,\ b$$
have been used in \cite{Tak11} for the construction of $\Delta$.
Some purposes of them are the following:

\begin{itemize}

\item  
$\lambda$ is concerned with rates
of growth of derivatives (see Proposition \ref{geo}II(a));

\item $\alpha\ll\lambda$ determines the speed of recurrence of critical points (see Proposition \ref{geo}II(b));


\item $\delta\ll1$ is the one in Lemma \ref{hyp} which determines the size of the critical region.

\end{itemize}

The $\alpha,\ \delta,\ b$ have been chosen in this order. In this paper we will shrink $\delta$ if necessary, at the expense of reducing $b$.
The letter $C$ will be used to denote generic positive constants independent
of $\alpha$, $\delta$, $b$.
Set $\kappa_0=C_0^{-10}$, where $C_0$ is the constant in Sect.\ref{initial}.

\subsection{Geometry of the unstable manifold and properties of critical orbits}\label{geometry}
We recall the properties of maps in $\{f_a\colon a\in\Delta\}$ as far as we need them.
Let $\mathcal C^{(0)}$ denote the closure of $I(\delta)$ and for $k\geq0$ write $f^kR=R_k$.
By an {\it  h-curve} we mean a compact, nearly horizontal $C^2$-curve in $\mathbb R^2$ such that the slopes of its tangent directions
are $\leq\sqrt{b}$ and the curvature is everywhere $\leq 1$.
We use the symbol $\approx$ to mean that both terms are equal up to a constant independent of $b$.

 \begin{prop}\label{geo} 
The following holds for $f\in\{f_a\colon a\in\Delta\}$:
\medskip


\noindent {\bf I (Geometry of the unstable manifold near the critical region)} There exist a nested sequence $\mathcal C^{(0)}\supset\mathcal
C^{(1)} \supset\mathcal C^{(2)}\supset\cdots$ and a countable set $\mathcal C$ in $W^u\cap I(\delta)$ such that the
following holds for $k=0,1,2,\ldots$:
\begin{itemize}

\item [(Ia)] $\mathcal C^{(k)}\subset R_k$, and $\mathcal C^{(k)}$ has a finite number of components
called $\mathcal Q^{(k)}$ each one of which is diffeomorphic to a
rectangle. The boundary of $\mathcal Q^{(k)}$ is made up of two
$C^2(b)$-curves of $\partial R_k$ connected by two vertical lines:
the horizontal boundaries are $\approx\min(2\delta,\kappa_0^{k})$
in length, and the Hausdorff distance between them is $\mathcal
O(b^{\frac{k}{2}})$;

\item [(Ib)] On each horizontal boundary $\gamma$ of each component $\mathcal
Q^{(k)}$ of $\mathcal C^{(k)}$, there is a unique element of $\mathcal C$ located
within $\mathcal O(b^{\frac{k}{4}})$ of the midpoint of $\gamma$;

\item [(Ic)] $\mathcal C^{(k)}$ is related to $\mathcal C^{(k-1)}$ as
follows: $\mathcal Q^{(k-1)}\cap R_{k}\neq\emptyset$, and has at most finitely many
components, each of which lies between two $C^2(b)$ subsegments of
$\partial R_{k}$ that stretch across $\mathcal Q^{(k-1)}$ as shown
in FIGURE 2. Each component of $\mathcal Q^{(k-1)}\cap R_{k}$
contains exactly one component of $\mathcal C^{(k)}$;

\item [(Id)] 
$\mathcal C=\bigcup_{k\geq0}\Xi^{(k)}$, where $\Xi^{(k)}$ denotes the collection
of elements of $\mathcal C$ on the horizontal boundaries of $\bigcup_{j=0}^k\mathcal C^{(j)}$.
Elements of $\mathcal C$ are called {\it critical points}.
\end{itemize}

\noindent {\bf II (Properties of critical points)} For each $\zeta\in\mathcal C$ the following holds:

\begin{itemize}

\item[(IIa)]
Let $\gamma$ be an $h$-curve in $I(\delta)$ which is tangent to $T_\zeta W^u$.
For $x\in\gamma\setminus\{\zeta\}$ let $t(\gamma;x)$ denote any unit vector tangent to $\gamma$ at $x$.
 There exists an integer $p(x)>0$ such that  
\begin{equation}\label{bind}\displaystyle{\|D_xf^{p(x)}t(\gamma;x)\|\geq 
e^{\frac{\lambda}{3}p(x)}}\ \text{ and }\ s
(D_xf^{p(x)}t(\gamma;x))\leq\sqrt{b};\end{equation}

\item[(IIb)]
$f^n\zeta\cap\mathcal C^{([\alpha n])}=\emptyset$ for every $n\geq1$;

\item[(IIc)] Let $k\geq0$ be an integer, and
let $\zeta_0$, $\zeta_1$ be critical points such that $\zeta_0\in\partial\mathcal Q^{(k)}$ and
$\zeta_1\in\partial\mathcal C^{(k+1)}\cap \mathcal Q^{(k)}$.
Then $|\zeta_0-\zeta_1|=\mathcal O(b^{\frac{k}{2}}) $.

\end{itemize}





\end{prop}
For items I, (IIa), (IIb), see \cite[Proposition 5.4, Proposition 5.2, Corollary 5.13]{Tak11} respectively.
Item (IIa) asserts that the loss of the magnitude of derivatives 
due to the folding behavior occurring inside $I(\delta)$ is recovered in time $p(x)$, and the slope is restored to a horizontal one. Item (IIb) bounds the speed of recurrence of each critical point under forward iteration.

Let us see why Item (IIc) holds. Let $n\geq1$ be an integer and $\gamma$ a $C^2(b)$-curve in $W^u\cap I(\delta)$. 
A point $\zeta\in\gamma$ is called a critical point of order $n$ on $\gamma$ if:

\begin{itemize}
\item $\|D_{f\zeta}f^k\|\geq 1$ for every $1\leq k\leq n$;

\item for any one-dimensional subspace $V$ of $T_{f\zeta}\mathbb R^2$ other than $T_{f\zeta}W^u$,
$\|D_{f\zeta} f|{T_{f\zeta}W^u}\|>\|D_{f\zeta}f|V\|$.
\end{itemize}

For each $n\geq1$ and a $C^2(b)$-curve $\gamma$ in $W^u\cap I(\delta)$, there is at most one critical point of order $n$
on $\gamma$ (c.f. \cite[Remark 2.4, Sect.2.4, Sect.2.5]{Tak11}). We call this property \emph{the uniqueness of critical points on $C^2(b)$-curves}.

Let $\gamma_0$ denote the horizontal boundary of $\mathcal C^{(k)}$ containing $\zeta_0$.
Let $\gamma_1$ denote the $C^2(b)$-curve in $\partial R_{k+1}\cap\mathcal C^{(k)}$ containing 
containing $\zeta_1$.
By Proposition \ref{geo}(Ia), the lengths of $\gamma_0$, $\gamma_1$ are $\approx\min(2\delta,\kappa_0^k)$, and the Hausdorff distance between them is 
$\mathcal O(b^{\frac{k}{2}})$. 

Each critical point in $\mathcal C$ has been constructed as a limit of a sequence of critical points of order $n$ $(n=1,2,\ldots)$.
In particular, there exist a critical point $\zeta_\epsilon^{(k)}$ of order $k$ on $\gamma_\epsilon$,
such that $|\zeta_\epsilon-\zeta_\epsilon^{(k)}|=\mathcal O(b^k)$ $(\epsilon=0,1)$.
By \cite[Lemma 3.1]{Tak12b}, there exists a critical point $\zeta$ of order $k$ on $\gamma_1$ such that
$|\zeta_0^{(k)}-\zeta|=\mathcal O(b^{\frac{k}{2}})$. The uniqueness of critical points on $C^2(b)$-curves yields $\zeta=\zeta_1^{(k)}$.
Hence we obtain $|\zeta_0-\zeta_1|\leq
|\zeta_0-\zeta_0^{(k)}|+|\zeta_0^{(k)}-\zeta_1| =\mathcal O(b^{\frac{k}{2}}) $.

\begin{figure}
\begin{center}
\includegraphics[height=6cm,width=10cm]
{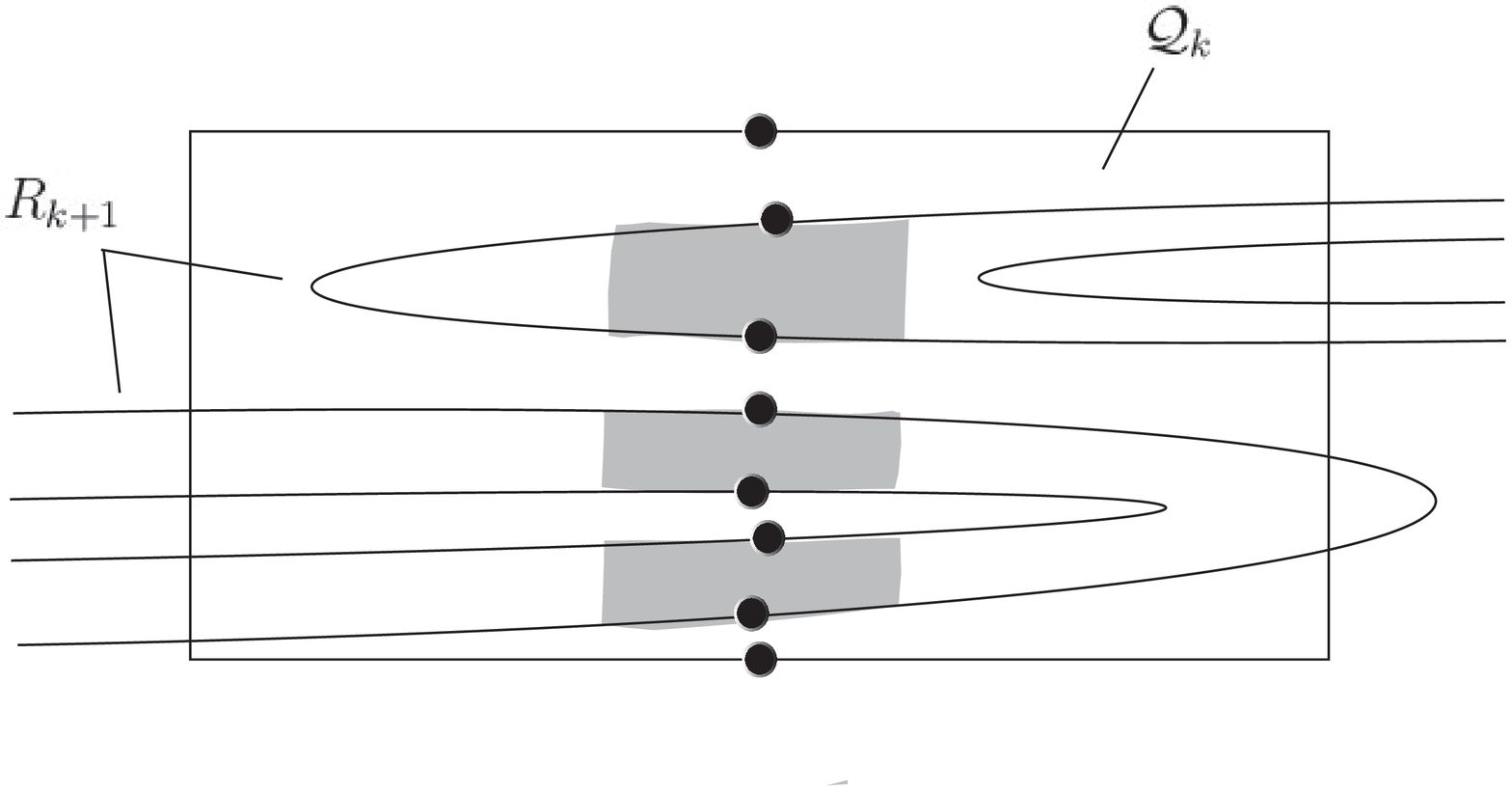}
\caption{The shaded regions denote components of $\mathcal C^{(k+1)}$ in $\mathcal Q^{(k)}$.The dots denote the critical points.}
\end{center}
\end{figure}

\subsection{Bound/free structure}\label{shrink}
In order to quantify the recurrence of critical pints to the set $\mathcal C$ we introduce a strictly decreasing sequence of sets as follows.
For $k\geq1$ let $\mathcal Q^{(k)}$ denote any component of $\mathcal C^{(k)}$.
Let $\zeta_0=(x_0,y_0)$ denote the critical point on the upper horizontal
boundary of $\mathcal Q^{(k)}$. 
Consider the two vertical lines $\{(x_0-r,y)\colon |y|\in\sqrt{b}\}$
and $\{(x_0+r,y)\colon |y|\in\sqrt{b}\}$.
Let $\mathcal Q^{(k)}(\delta^{\frac{k}{2}})$ denote the closed region 
bordered by these two lines and the horizontal boundaries of $\mathcal Q^{(k)}$.
Note that, by Proposition \ref{geo}, $\mathcal Q^{(k)}(\delta^{\frac{k}{2}})$ contains the two 
critical points on the horizontal boundaries of $\mathcal Q^{(k)}$.
Set $$\mathcal B^{(k)}(\delta^{\frac{k}{2}})=\bigcup\mathcal Q^{(k)}(\delta^{\frac{k}{2}}),$$ where the union runs over
all components of $\mathcal C^{(k)}.$
By Proposition \ref{geo}II(c),
$\mathcal B^{(k)}(\delta^{\frac{k}{2}})$ is strictly decreasing in $k$.

Let $x\in \Omega\cap I(\delta)$ and $\nu>1$ an integer.
We say $x$ is {\it controlled up to time $\nu$}  if 
\begin{equation}\label{contro}
f^n x\notin \mathcal B^{(n)}(\delta^{\frac{n}{2}})\ \ \text{for every }1\leq
n<\nu.\end{equation} 
If $x$ is controlled up to time $\nu$ and not so up to time $\nu+1$,
namely $f^\nu x\in \mathcal B^{(\nu)}(\delta^{\frac{\nu}{2}})$ holds, then we
say $x$ makes a {\it close return} at time $\nu$, and call $\nu$ a
close return time of $x.$
We say $x$ is {\it controlled} if it is controlled up to time $\nu$ for every $\nu>1$.

To a controlled point $x\in\Omega\cap I(\delta)$ we associate a sequence of integers 
\begin{equation}\label{integer}0<n_1<n_1+p_1\leq
n_2<n_2+p_2\leq n_3<\cdots\end{equation}
inductively as follows: 
$n_1$ is the smallest $n>0$ with $f^nx\in I(\delta)$.
By Lemma \ref{hyp}(b), $s(w_{n_1}(x))\leq\sqrt{b}$ holds.
Given $n_i>0$ $(i\geq1)$ with $f^{n_i}x\in I(\delta)$ and 
$s(w_{n_i}(x))\leq\sqrt{b}$, 
let $k_i$ denote the maximal $k\in[0,n_i]$ with $f^{n_i}x\in\mathcal C^{(k)}$.
Let $\mathcal Q^{(k_i)}$ denote the component of $\mathcal C^{(k_i)}$ containing $f^{n_i}x$, and
$\zeta_i$ the 
critical point on the upper horizontal boundary of $ \mathcal Q^{(k_i)}$.
Write $f^{n_i}x=(x_0,y_0)$ and $\zeta_i=(x_1,y_1)$.
By Proposition \ref{geo}(Ia)(Ib), we have $|y_0-y_1|\leq  2\sqrt{b}|x_0-x_1|+Cb^{\frac{k_i}{2}}$.
If $k_i<n_i$, then additionally using Proposition \ref{geo}(Ic) we have $|x_0-x_1|\geq 
C\kappa_0^{k_i+1},$ where $C>0$ is independent of $b$. Otherwise, i.e., if $k_i=n_i$,
the condition $f^{n_i} x\notin \mathcal Q^{(n_i)}(\delta^{\frac{n_i}{2}})$
gives $|x_0-x_1|\geq\delta^{\frac{n_i}{2}}.$
Hence, in either of the two cases
$|y_0-y_1|/|x_0-x_1|\leq 3\sqrt{b}$ holds.
This implies that there exists an $h$-curve $\gamma_i$ which is tangent to both
$T_{\zeta_i}W^u$ and $w_{n_i}(x)$. 
Let $p_i=p(f^{n_i}x,\zeta_i,\gamma_i)$ denote the integer determined by Proposition \ref{geo}(IIa).
Define $n_{i+1}$ to be the smallest $n\geq n_i+p_i$ with $f^nx\in I(\delta).$
This finishes the definition of the sequence in \eqref{integer}.

The sequence in \eqref{integer} decomposes the forward orbit of $x$ into alternative bound and free segments,
corresponding to time intervals $[n_i+1,n_i+p_i-1]$ and $[n_i+p_i,n_{i+1}]$, during which we refer
to the orbit of $x$ as being {\it bound} and {\it free} respectively.

\begin{lemma}\label{conduir} If $x\in\Omega\cap I(\delta)$ is controlled up to time $n$, 
and $f^nx$ is free, then 
$\Vert w_{n}(x)\Vert\geq \delta e^{\frac{\lambda}{3}(n-1)}.$
If, in addition $f^nx\in I(\delta)$, then the constant $\delta$ can be dropped.
\end{lemma}
\begin{proof}
Let $0<n_1<n_1+p_1\leq n_2<\cdots$ be the sequence of integers defined as above.
The following derivative estimates follow from Lemma \ref{hyp} and Proposition \ref{geo}(IIa): 
\begin{equation}\label{estimates}
\|w_{n_i+p_i}(x)\|\geq e^{\frac{\lambda}{3}p_i}\|w_{n_i}(x)\|\ \ \text{and}\ \
\|w_{n_{i+1}}(x)\|  \geq e^{\hat \lambda(n_{i+1}-n_i-p_i)}\|w_{n_i+p_i}(x)\|.\end{equation}
The first estimate of Lemma \ref{conduir} follows from \eqref{estimates} and the Chain Rule.
For the last one we additionally use the last statement of Lemma \ref{hyp}(a).
\end{proof}

\subsection{Proof of Proposition \ref{dich}.}\label{22}
We are in position to finish the proof of Proposition \ref{dich}.

\begin{lemma}\label{segment}
Let $\nu\geq1$, and let $\ell$ be a compact $C^1$ curve in $R$ with 
${\rm length}(\ell)\leq 2\delta^{\frac{\nu}{2}}$.
If $1\leq n<2\nu$ and $ I(\delta)\cap f^n\ell\neq\emptyset$,
then 
$f^n\ell\subset R$ and ${\rm length}(f^n\ell)\leq \delta^{\frac{\nu}{3}}$. 
\end{lemma}

\begin{proof}
Take a point $x\in f^{-n}(I(\delta)\cap f^n\ell)$.
For $j\in[0,n]$ let $A_j$ denote the connected component of $\ell\cap f^{-j}R$ containing $x$.
Since points which escape out of $R$ do not return to $R$ under forward iteration as in Lemma \ref{nws}(a), $\{A_j\}_j$ is decreasing in $j$.
Lemma \ref{nws}(a) also implies
$\bigcup_{i=0}^nf^iA_n\subset R$. From this and $\|Df\|\leq 5$ on $R$ we have 
\begin{equation}\label{length}{\rm length}(f^nA_n)\leq 5^n\cdot{\rm length}(A_n)\leq 5^{n}\cdot{\rm length}(\ell)
\leq 5^{2\nu}\cdot 2\delta^{\frac{\nu}{2}}<1/2.\end{equation}

We claim $f^n\ell\subset R$. This clearly holds for the case $\ell=A_n$. Consider the case $\ell\neq A_n$.
Let $k\in[0,n]$ denote the minimal such that
$A_k=A_n$. Since $A_0=\ell$, the case $k=0$ is already covered.
 Suppose $k>0$. Then we have $A_{k-1}\neq A_n$. 
Recall that  $\Theta_\infty$ denotes the compact lenticular domain bordered by the parabola in $W^s(Q)$ and one of the two boundary curves of $R$
formed by $W^u$. Since $f^{-1}(\mathbb R^2\setminus R)\cap R={\rm int}\Theta_\infty$, 
$f^kA_{k-1}$ intersects $\Theta_\infty$.
Since $x\in A_{k-1}$ and $f^kx\in R\setminus{\rm int}\Theta_\infty$,
$f^kA_k$ intersects $\Theta_\infty\cap W^s(Q)$.
Since $f^nx\in f^nA_k$ and $f^nx\in I(\delta)$, and since
any forward iterates of $\Theta_\infty\cap W^s(Q)$ are contained in
the two curves in $W^s(Q)$ forming the boundary of $R$,
we have ${\rm length}(f^nA_k)={\rm length}(f^nA_n)>1/2$. This yields a contradiction to \eqref{length}, and the claim holds.

From the above claim and Lemma \ref{nws}(a),
$\bigcup_{i=0}^{n-1}f^i\ell\subset R$ holds. This yields 
\begin{equation*}{\rm length}(f^n\ell)\leq 5^n\cdot{\rm length}(\ell)
\leq 5^{2\nu}\cdot 2\delta^{\frac{\nu}{2}}<\delta^{\frac{\nu}{3}}.\qedhere\end{equation*}
\end{proof}
\medskip

\noindent{\it Proof of Proposition \ref{dich}.}
Let $x\in\Omega$. 
Define nonnegative integers
$\nu_0,\nu_1,\ldots$ inductively as follows: 
$\nu_0$ is the smallest $n\geq0$ with $f^nx\in I(\delta)$.
Given $\nu_0,\ldots,\nu_{l},$ define
$\nu_{l+1}$ to be the close return time of
$f^{\nu_0+\nu_1+\cdots+\nu_{l}}x\in I(\delta)$. 
If  $\nu_0,\ldots,\nu_{l}$ are defined and 
$f^{\nu_0+\nu_1+\cdots+\nu_{l}}x\in I(\delta)$ is controlled, and thus $\nu_{l+1}$ is undefined,
then set $\bar\nu=\nu_0+\nu_1+\cdots+\nu_{l}$.

If $\bar\nu$ is defined, then $f^{\bar\nu}x$ is controlled. By Lemma \ref{conduir},
for infinitely many $n\geq1$,
$$\|w_n(f^{\bar \nu}x)\|
\geq \delta e^{\frac{\lambda}{3} (n-1)}\geq e^{\frac{\log2}{4}(n-1)}.$$
Hence Proposition \ref{dich}(a) holds.

If $\bar\nu$ is undefined, then we end up with the infinite sequence
$\{\nu_l\}_{l=0}^\infty$ of nonnegative integers. By definition, each $\nu_{l+1}$ is a close return time of 
$f^{\nu_0+\nu_1+\cdots+\nu_{l}}x$. By Lemma \ref{conduir}, 
$$\|w_{\nu_{l+1}}(f^{\nu_0+\nu_1+\cdots+\nu_{l}}x)\|
\geq e^{\frac{\lambda}{3}(\nu_{l+1}-1)}\geq e^{\frac{\log2}{4}(\nu_{l+1}-1)}.$$
Hence Proposition \ref{dich}(b-i) holds.

We have $\nu_1>0$. To prove Proposition \ref{dich}(b-ii) 
it is left to show $\nu_{l+1}\geq 2\nu_{l}$.
Set $y=f^{\nu_{1}+\cdots+\nu_{l-1}+\nu_l}x$.
We show that if $1\leq n\leq 2\nu_l$, then $n$ is not a close return time of $y$.
We derive a contradiction assuming $1\leq n\leq 2\nu_l$ and $n$ is a close return time of $y$.

Since $\nu_l$ is a close return time of $f^{\nu_{1}+\cdots+\nu_{l-1}}x$,  
$y\in \mathcal B^{(\nu_l)}(\delta^{\frac{\nu_l}{2}})$.
Let $\zeta$ denote the critical point on the upper horizontal boundary of the component of 
$\mathcal B^{(\nu_l)}(\delta^{\frac{\nu_l}{2}})$ 
containing $y$.
Take a compact $C^1$ curve $\ell$ in $R$ connecting $y$ and $\zeta$,
with ${\rm length}(\ell_0)\leq 2\delta^{\frac{\nu_l}{2}}$. Since $f^ny\in I(\delta)\cap f^n\ell$,
by Lemma \ref{segment},
\begin{equation}\label{distan}
|f^ny-f^n\zeta|\leq{\rm length}(f^n\ell)\leq\delta^{\frac{\nu_l}{3}}.\end{equation}

Since $n$ is a close return time of $y$, $f^ny\in\mathcal B^{(n)}(\delta^{\frac{n}{2}})$.
Let $\zeta'$ denote the critical point on the upper horizontal boundary of the component of 
$\mathcal B^{(n)}(\delta^{\frac{n}{2}})$ containing $f^ny$. Then $|\zeta'-f^ny|\leq 2\delta^{\frac{n}{2}}$ holds.
Choose a sequence $\zeta^{(n)},\zeta^{(n-1)},\ldots,\zeta^{([\alpha n])}$ of critical points such that
$\zeta^{(n)}=\zeta'$, and the following holds: $\zeta^{(i)}\in\partial\mathcal C^{(i)}$ for every $i\in\{n,n-1,\ldots,[\alpha n]\}$,
and if $\mathcal Q^{(i)}$ denotes the component of $\mathcal C^{(i)}$ whose horizontal boundary contains 
$\zeta^{(i)}$, then $\mathcal Q^{(n)}\subset\mathcal Q^{(n-1)}\subset\cdots\subset\mathcal Q^{([\alpha n])}$.
Proposition \ref{geo}II(b) gives $f^n\zeta\notin\mathcal C^{[\alpha n]}$.
Hence $|f^n\zeta-\zeta^{[\alpha n]}|\geq\kappa_0^{\alpha n}$.
Proposition \ref{geo}II(c) yields
$$|\zeta^{([\alpha n])}-\zeta'|\leq\sum_{i=[\alpha n]}^{n-1}
|\zeta^{(i)}-\zeta^{(i+1)}|\leq C\sum_{i=[\alpha n]}^nb^{\frac{i}{2}}\leq Cb^{\frac{\alpha n}{3}}.$$
Putting these estimates together, and then using the fact that $\kappa_0$ is independent of $\delta$ and $b$ we obtain
\begin{align*}
|f^ny-f^n\zeta|&\geq |f^n\zeta-\zeta^{([\alpha n])}|-|\zeta^{([\alpha n])}-\zeta'|-|\zeta'-f^ny|\\
&\geq \kappa_0^{\alpha n}-Cb^{\frac{\alpha n}{3}}-2\delta^{\frac{n}{2}}\geq\kappa_0^{2\alpha n}\geq\kappa_0^{4\alpha \nu_l},\end{align*}
which is a contradiction to \eqref{distan}.
This completes the proof of Proposition \ref{dich}(b-ii).
\qed
\medskip

\section{Uniqueness of tangencies at the first bifurcation parameter}
In this last section we show the uniqueness of tangencies at the first bifurcation parameter $a^*$.
Recall that a point $x\in\mathbb R^2$ is called {\it homoclinic} 
if there exists a saddle $p$ such that $x\neq p$ and $x\in W^s(p)\cap W^u(p)$. 
It is called {\it heteroclinic} if there exist
two distinct saddles $p$, $q$ such that $x\in W^s(p)\cap W^u(q)$.
A homoclinic or heteroclinic point $x$ is called {\it transverse} if the two invariant manifolds
through $x$ intersect each other transversely.  
Recall that $\zeta_0$ is the point of tangency near $(0,0)$ between the stable and unstable manifolds of the fixed points of $f_{a^*}$ (see FIGURE 1).
\begin{theoremb}
Any homoclinic or heteroclinic point of $f_{a^*}$ other than $f_{a^*}^n\zeta_0$ $(n\in\mathbb Z)$ is transverse.
\end{theoremb}

Let $\varepsilon>0$ and define
$$\Omega_{a^*}(\varepsilon):=\{x\in \Omega_{a^*}\colon |f_{a^*}^nx-\zeta_0|\geq\varepsilon\ \text{\rm for every $n\in\mathbb Z$}\},$$
which is a compact and $f_{a^*}$-invariant set.

\begin{prop}\label{hyperbolicity}
For any $\varepsilon>0$, 
$\Omega_{a^*}(\varepsilon)$
is a hyperbolic set.\end{prop}

Recall that $K_{a^*}=\{x\in\mathbb R^2\colon\text{$\{f_{a^*}^nx\}_{n\in\mathbb Z}$ is bounded}\}$.
Any homoclinic or heteroclinic point of $f_{a^*}$ is contained in $K_{a^*}$, 
and so in $\Omega_{a^*}$ because $K_{a^*}=\Omega_{a^*}$ from the proof of Lemma \ref{K}.
Hence, any homoclinic or heteroclinic point of $f_{a^*}$ other than 
$f_{a^*}^n\zeta_0$ $(n\in\mathbb Z)$ is contained in $\bigcup_{\varepsilon>0}\Omega_{a^*}(\varepsilon)$.
Theorem B follows from Proposition \ref{hyperbolicity}.
\medskip

To prove Proposition \ref{hyperbolicity} 
we need two results from \cite{SenTak11}:

\begin{itemize}
\item {\rm {\bf Unstable subspace} (\cite[Proposition 4.1]{SenTak11})}
at each point $x\in \Omega_{a^*}$ there exists a one-dimensional subspace $E^u(x)\subset T_x\mathbb R^2$ such that
\begin{equation}\label{eu}
\limsup_{n\to\infty}\frac{1}{n}\log\|D_xf_{a^*}^{-n}|E^u(x)\|<0;\end{equation}

\item {\rm {\bf Bound period} (\cite[Proposition 2.5]{SenTak11})} for each $x\in I(\delta)\cap(\Omega_{a^*}\setminus\{\zeta_0\})$ there exists $p(x)\in\mathbb N$ such that
\begin{equation}\label{delest}
\|D_xf_{a^*}^{p(x)}|E^u(x)\|\geq e^{\frac{\lambda}{3}p(x)}\ \ \text{and}\ \
s(E^u(f_{a^*}^{p(x)}x))\leq\sqrt{b}.\end{equation}
Here, $s(E^u(f_{a^*}^{p(x)}x))=s(v)$ and $v$ is a vector
spanning $E^u(f_{a^*}^{p(x)}x)$.
\end{itemize}
\medskip

\noindent{\it Remark.} Since $f^{-1}$ expands area, the one-dimensional subspace of $T_x\mathbb R^2$ with the property in \eqref{eu} is unique.
\medskip

\noindent{\it Remark.} Any ergodic measure of $f_{a^*}$ is a hyperbolic measure [\cite{CLR08} and Theorem A], and $E^u(\cdot)$ coincides with the subspace 
in the Oseledec decomposition corresponding to positive Lyapunov exponents.\medskip

\noindent{\it Proof of Proposition \ref{hyperbolicity}.} 
We shall find $\lambda>1$ and $N>0$ such that at each $x\in\Omega_{a^*}(\varepsilon)$,
\begin{equation}\label{cocycle}\|D_xf_{a^*}^n|E^u(x)\|\geq \lambda^n \text{ for every }n\geq N.\end{equation}
From \cite[Lemma 7.3]{WanYou01} and $|\det Df_{a^*}|<1$,
it follows that $\Omega_{a^*}(\varepsilon)$ is a hyperbolic set.

 Similarly to \eqref{integer},
to the forward orbit of each $x\in\Omega_{a^*}(\varepsilon)$ we associate a sequence $n_0\leq n_1<n_1+p_1<n_2<n_2+p_2<\cdots$
of integers inductively as follows.
First, define $n_0=\inf\{n\geq0\colon\text{$s(E^u(f_{a^*}^nx))\leq\sqrt{b}$}\}.$
Then, define $n_1$ to be the smallest $n\geq n_0$ with $f_{a^*}^{n}x\in I(\delta)$. 
Given $n_k$ with $f_{a^*}^{n_k}x\in I(\delta)$,
Define
$p_k=p(f_{a^*}^{n_k}x)$.
Define $n_{k+1}$ to be the smallest $n> n_k+p_k$ with $f_{a^*}^nx\in I(\delta)$, and so on.

\begin{lemma}\label{start}
There exists $N=N(\varepsilon)$ such 
that for all $x\in\Omega_{a^*}(\varepsilon)$, $n_0=n_0(x)\leq N$.
\end{lemma}
\begin{proof}
Let $\alpha_1^+$ denote the connected component of $R_{a^*}\cap
W^s(P_{a^*})$ containing 
$P_{a^*}$.
Let $\alpha_1^-$ denote the connected component of $R_{a^*}\cap
f^{-1}\alpha_1^+$ not containing $P_{a^*}$.
Let $\Theta$ denote the compact domain bordered by $\alpha_1^+$, $\alpha_1^-$ and the two boundary curves of $R$ formed by $W^u$.
Let $\partial^sR$ denote the union of the two boundary curves of $R$ formed by $W^s(Q_{a^*})$, and let
$B(\delta^3)=\{x\in\mathbb R^2\colon{\rm dist}(x,\partial^sR)\leq\delta^3\}.$

Let $x\in\Omega_{a^*}$.
If $x\notin B(\delta^3)$, then let $\hat n(x)$ denote the minimal $n\geq0$ with
$f_{a^*}^nx\in\Theta$. 
Clearly, there exists $C(\delta)>0$ independent of $x$ such that
$\hat n(x)\leq C(\delta)$.
By  \cite[Proposition 4.1]{SenTak11}, $s(E^u(f_{a^*}^{\hat n(x)}x))\leq\sqrt{b}$, and so 
$n_0(x)\leq\hat n(x)\leq C(\delta)$.
If $x\in B(\delta^3)$, then there exists $m>0$ such that $y:=f_{a^*}^{-m}x\in I(\delta)$
and $f_{a^*}^iy\in   B(\delta^3)$ for $1\leq i\leq m$.
If $p(y)\leq m$, then $s(E^u(x))\leq\sqrt{b}$,
and so $n_0(x)=0$. If $p(y)>m$, then $n_0(x)\leq p(y)-m$ because
$s(E^u(f_{a^*}^{p(y)}y))\leq\sqrt{b}$ .
Set $$p_{\rm sup}(\varepsilon)=\sup\{p(x)\colon x\in\Omega_{a^*}(\varepsilon)\cap I(\delta)\}.$$
By a result of \cite[Sect.2]{SenTak11}, $p_{\rm sup}(\varepsilon)<\infty$.
Set $N=\max\{C(\delta),p_{\rm sup}(\varepsilon)\}$.
\end{proof}

\begin{lemma}\label{exponential}
For any $\varepsilon>0$ there exists $C>0$ such that 
for each $x\in\Omega_{a^*}(\varepsilon)$,
$$\|D_{f_{a^*}^{n_0(x)}x}f_{a^*}^n|E^u(f_{a^*}^{n_0(x)}x)\|\geq Ce^{\frac{\lambda}{4}(n-n_0(x))} \text{ for every }n\geq n_0(x).$$
\end{lemma}

\begin{proof}
Set
$C_0=\inf\{|\det D_yf_{a^*}|/\|D_yf_{a^*}\|\colon {y\in\Omega_{a^*}}\}$. Observe that $C_0\in(0,1)$.
For any unit vector $u$ at $y\in\Omega_{a^*}$ we have 
\begin{equation}\label{u}\|D_yf_{a^*}u\|\geq|\det D_yf_{a^*}|/\|D_yf_{a^*}\|
\geq C_0.\end{equation}

Let $x\in\Omega_{a^*}(\varepsilon)$ and $n\geq n_0$. If 
$n\notin\cup_{k=1}^\infty(n_k,n_k+p_k)$, 
then by Lemma \ref{hyp} and \eqref{delest}, $$\|D_{f_{a^*}^{n_0}x}f_{a^*}^{n-n_0}E^u(f_{a^*}^{n_0}x)\|\geq \delta e^{\frac{\lambda}{3}(n-n_0)}.$$
If $n\in(n_k,n_k+p_k)$ holds for some $k\geq1$, then
since $n-n_k<p_k\leq p_{\rm sup}(\varepsilon)$,
\begin{align*}
\|D_{f_{a^*}^{n_0}x}f_{a^*}^n|E^u(f_{a^*}^{n_0}x)\|&=
\|D_{f_{a^*}^{n_k}x}f_{a^*}^{n-n_k}|E^u(f_{a^*}^{n_k}x)\|\cdot\|D_{f_{a^*}^{n_0}x}f^{n_k-n_0}|E^u(f_{a^*}^{n_0}x)\|\\
&\geq 
C_0^{p_{\rm sup}(\varepsilon)}\cdot e^{\frac{\lambda}{3}(n_k-n_0)}.\end{align*}
To estimate the first factor of the right-hand-side we have used \eqref{u}. We have
\begin{align*}
 \frac{e^{\frac{\lambda}{4}(n-n_0)}}{e^{\frac{\lambda}{4}(n-n_0)}}\cdot e^{\frac{\lambda}{3}(n_k-n_0)}
\geq \frac{e^{\frac{\lambda}{4}(n-n_0)}}{e^{\frac{\lambda}{4}(n_k+p_k-n_0)}}\cdot e^{\frac{\lambda}{3}(n_k-n_0)}\geq \frac{e^{\frac{\lambda}{4}(n-n_0)}}{e^{\frac{\lambda}{4}p_k}}
\geq \frac{e^{\frac{\lambda}{4}(n-n_0)}}{e^{\frac{\lambda}{4}p_{\rm sup}(\varepsilon)}}.
\end{align*}
Set
$C=\min\{\delta, (C_0e^{-\frac{\lambda}{4}})^{p_{\rm sup}(\varepsilon)}\}.$ 
\end{proof}
 \eqref{cocycle} now follows from Lemma \ref{start} and Lemma \ref{exponential}. \qed

\subsection*{Acknowledgments}
Partially supported by the Grant-in-Aid for Young Scientists (B) of the JSPS, Grant No.23740121
and the Keio Gijuku Academic Development Funds 2013.
I thank Masayuki Asaoka for pointing out an error in the proof of Lemma \ref{zeropesin} in a previous version of the paper.

\bibliographystyle{amsplain}

\end{document}